\documentclass{amsart}
\usepackage{amsmath,amssymb,amsxtra,amsfonts,amsthm,comment,url,microtype}

\usepackage{tikz}
\usetikzlibrary{arrows,backgrounds,decorations,automata}

\theoremstyle{plain}
\newtheorem{theorem}{Theorem}
\newtheorem{lemma}[theorem]{Lemma}
\newtheorem{corollary}[theorem]{Corollary}
\newtheorem{proposition}[theorem]{Proposition}

\newtheorem{question}[theorem]{Question}

\theoremstyle{definition}
\newtheorem{definition}[theorem]{Definition}

\newcommand{\B}{\mathbb}
\newcommand{\C}{\mathcal}

\newcommand{\eps}{\varepsilon}

\newcommand{\gG}{\Gamma}

\DeclareMathOperator{\trdeg}{tr.deg.\, }

\newcounter{nootje}
\begin{document}

\title{An asymptotic approach in Mahler's method}

\author{Michael Coons}
\address{School of Math.~and Phys.~Sciences\\
University of Newcastle\\
Callaghan\\
Australia}
\email{Michael.Coons@newcastle.edu.au}

\thanks{The research of M.~Coons was supported by ARC grant DE140100223.}


\date{\today}

\begin{abstract} We provide a general result for the algebraic independence of Mahler functions by a new method based on asymptotic analysis. As a consequence of our method, these results hold not only over $\B{C}(z)$, but also over $\B{C}(z)(\C{M})$, where $\C{M}$ is the set of all meromorphic functions. Several examples and corollaries are given, with special attention to nonnegative regular functions.
\end{abstract}

\maketitle

\section{Introduction}

Mahler's method is a method in number theory wherein one answers questions surrounding the transcendence and algebraic independence of both functions $F(z)\in\B{C}[[z]]$, which satisfy the functional equation \begin{equation}\label{MFE}a_0(z)F(z)+a_1(z)F(z^k)+\cdots+a_d(z)F(z^{k^d})=0\end{equation} for some integers $k\geqslant 2$ and $d\geqslant 1$ and polynomials $a_0(z),\ldots,a_d(z)\in\B{C}[z]$, and their special values $F(\alpha)$, typically at algebraic numbers $\alpha$. Functions $F(z)$ satisfying a functional equation of the type in \eqref{MFE} are called {\em $k$-Mahler} (or simply {\em Mahler}, when $k$ is understood); the minimal $d$ for which $F(z)$ satisfies \eqref{MFE} is called the {\em degree} of $F(z)$. Such functions can be considered in a vector setting as well, wherein one considers a vector of functions ${\bf F}(z)=[F_1(z),\ldots,F_d(z)]^T\in\mathbb{C}[[z]]^d$ for some integer $d\geqslant 1$ for which there is a matrix of rational functions ${\bf A}(z)\in\mathbb{C}(z)^{d\times d}$ and an integer $k$ such that \begin{equation}\label{vecmahl} {\bf F}(z)={\bf A}(z){\bf F}(z^k).\end{equation}

Questions and results concerning the transcendence of Mahler functions and their special values were studied in depth by Mahler in the late 1920s and early 1930s \cite{M1929, M1930a, M1930b}, though the study of special Mahler functions dates back to at least the beginning of the XXth century with the publication of Whittaker and Watson's classic text, ``A Course of Modern Analysis'' \cite[Section~5$\cdot$501]{WW1902}. Therein the Mahler function $\sum_{n\geqslant 0}z^{2^n}$ is presented as an example of a function having the unit circle as a natural boundary. 

Mahler's early results focused on degree-$1$ Mahler functions, his most famous result in this area being the transcendence of the Thue-Morse number $T(1/2)$, which is a special value of the function $T(z)$ satisfying $T(z)-(1-z)T(z^2)=0$. According to Waldschmidt \cite{W2009}, after Mahler's initial results his method was forgotten; the resurgence waited nearly forty years, following the publication of Mahler's paper ``Remarks on a paper of W.~Schwarz'' \cite{M1969} in 1969. Mahler's method was then extended by Kubota, Loxton, Ke.~Nishioka, Ku.~Nishioka, and van der Poorten among others; see \cite{K1977b, K1977, L1984, LvdP1976, LvdP1977, LvdP1977b, LvdP1978, LvdP1982, LvdP1988, KeN1984, KeN1985, KuN1982, KuN1990, KuN1996, KuN1996book}, though this list is certainly not exhaustive. Much of the continuing interest is connected with the fact that if the sequence $\{f(n)\}_{n\geqslant 0}$ is output by a deterministic finite automaton, then its generating function $F(z)=\sum_{n\geqslant 0}f(n)z^n$ is a Mahler function. 

Arguably, the most celebrated result in this area is due to Ku.~Nishioka \cite{KuN1990}, who proved that if $F_1(z),\ldots,F_d(z)$ are components of a vector of Mahler functions with algebraic coefficients satisfying \eqref{vecmahl}, then for all but finitely many algebraic numbers $\alpha$ in the common disc of convergence of $F_1(z),\ldots,F_d(z)$, we have $$\trdeg_{\B{Q}}\B{Q}(F_1(\alpha),\ldots,F_d(\alpha))=\trdeg_{\B{C}(z)}\B{C}(z)(F_1(z),\ldots,F_d(z)).$$ Ku.~Nishioka's result fully reveals the heart of Mahler's method, {\em one can obtain an algebraic independence result for the special values of Mahler functions by producing the result at the function level.} Of course, to gain full use of this theorem, one must produce a function-level result.

While there are several results concerning specific functions of degrees $1$ and $2$ (see in particular the recent work of Bundschuh and V\"a\"an\"anen \cite{B2012, B2013, BV2014, BV2015c, BV2015d, BV2015a, BV2015b}), there is a lack of general results for the algebraic independence of Mahler functions. For degree-$1$ Mahler functions, general results have been given by Kubota \cite{K1977} and Ke.~Nishioka \cite{N1984}, though the criteria they provide can be quite hard to check, making their results difficult to apply.

In this paper, we provide a general algebraic independence result for Mahler functions of arbitrary degree. Our result is based on properties of the eigenvalues of Mahler functions, a concept we recently introduced with Bell \cite{BCpre} in order to produce a quick transcendence test for Mahler functions. To formalise this notion here, suppose that $F(z)$ satisfies \eqref{MFE}, set $a_i:=a_i(1)$, and form the characteristic polynomial  of $F(z)$, $$p_F(\lambda):=a_0\lambda^d+a_1\lambda^{d-1}+\cdots+a_{d-1}\lambda+a_d.$$ In the above-mentioned work with Bell, we showed that if $p_F(\lambda)$ has $d$ distinct roots, then there exists an eigenvalue $\lambda_F$ with $p_F(\lambda_F)=0$, which is naturally associated to $F(z)$. We use the term `eigenvalue' to denote the root of a characteristic polynomial.

Our first result is the following.

\begin{theorem}\label{main} Let $k\geqslant 2$ be an integer, $F_1(z),\ldots,F_d(z)\in\B{C}[[z]]$ be $k$-Mahler functions convergent in the unit disc for which the eigenvalues $\lambda_{F_1},\ldots,\lambda_{F_d}$ exist, and let $\C{M}$ denote the set of meromorphic functions. If $k,\lambda_{F_1},\ldots,\lambda_{F_d}$ are multiplicatively independent, then $$\trdeg_{\B{C}(z)(\C{M})}\B{C}(z)(\C{M})(F_1(z),\ldots,F_d(z))=d.$$ In particular, the functions $F_1(z),\ldots,F_d(z)$ are algebraically independent over $\B{C}(z)$.
\end{theorem}

As the title of this paper suggests, our results are obtained by an asymptotic argument. Indeed, the ability to include meromorphic functions in Theorem~\ref{main} is a by-product of our method being analytic and not heavily dependent on algebra. Though our result adds general meromorphic functions in the context of algebraic independence, the comparison of Mahler functions with meromorphic functions is not new. B\'ezivin \cite{B1994} showed that a Mahler function that satisfies a homogeneous linear differential equation with polynomial coefficients is necessarily rational. Taking this further, in his thesis (and unpublished otherwise), Rand\'e \cite{R1992} proved that a Mahler function is either rational or has a natural boundary; see our paper with Bell and Rowland \cite{BCR2013} for a more recent proof of this result. 

While Theorem \ref{main} is quite general, if we focus on a certain subclass of Mahler functions, the nonnegative $k$-regular functions, we can remove the existence assumption on the eigenvalues $\lambda_{F_i}$ for $i=1,\ldots,d$. 

An integer-valued sequence $\{f(n)\}_{n\geqslant 0}$ is called {\em $k$-regular} provided there exist a positive integer $d$, a finite set of matrices $\{{\bf A}_0,\ldots,{\bf A}_{k-1}\}\subseteq \B{Z}^{d\times d}$, and vectors ${\bf v},{\bf w}\in \B{Z}^d$ such that $$f(n)={\bf w}^T {\bf A}_{i_0}\cdots{\bf A}_{i_s} {\bf v},$$ where $(n)_k={i_s}\cdots {i_0}$ is the base-$k$ expansion of $n$. The notion\footnote{Our definition is not the definition of Allouche and Shallit, though a result of theirs \cite[Lemma~4.1]{AS1992} gives the equivalence.} of $k$-regularity is due to Allouche and Shallit \cite{AS1992}, and is a direct generalisation of automaticity; in fact, a $k$-regular sequence that takes finitely many values can be output by a deterministic finite automaton. We call the generating function $F(z)=\sum_{n\geqslant 0}f(n)z^n$ of a $k$-regular sequence $\{f(n)\}_{n\geqslant 0}$, a {\em $k$-regular function} (or just {\em regular}, when the $k$ is understood). Establishing the relationship to Mahler functions, Becker \cite{pgB1994} proved that a $k$-regular function is also a $k$-Mahler function. 

In order to prove an algebraic independence result for regular functions, we prove the following result on the asymptotics of $k$-regular sequences.

\begin{theorem}\label{fNlogN} Let $k\geqslant 2$ be a integer and $\{f(n)\}_{n\geqslant 0}$ be a nonnegative integer-valued $k$-regular sequence, which is not eventually zero. Then there is a real number $\alpha_f\geqslant 1$ and a nonnegative integer $m_f$ such that as $N\to\infty$, $$\alpha_f^{-1}(1+o(1))\leqslant \frac{\sum_{n\leqslant N} f(n)}{N^{\log_k\alpha_f} \log^{m_f}N}\leqslant \alpha_f(1+o(1)).$$
\end{theorem}

We stress that Theorem \ref{fNlogN} provides the existence of a constant $\alpha_f$, which essentially takes the place of the Mahler eigenvalue for regular functions. We use these asymptotics to give the following result for $k$-regular functions.

\begin{theorem}\label{mainreg} Let $k\geqslant 2$ be an integer, $F_1(z),\ldots,F_d(z)\in\B{Z}_{\geqslant 0}[[z]]$ be $k$-regular functions with $F_i(z):=\sum_{n\geqslant 0}f_i(n)z^n$ for $i=0,\ldots,d$, and let $\C{M}$ denote the set of meromorphic functions. If the numbers $k,\alpha_{f_1},\ldots,\alpha_{f_d}$ are multiplicatively independent, where the $\alpha_{f}$ are provided by Theorem \ref{fNlogN}, then $$\trdeg_{\B{C}(z)(\C{M})}\B{C}(z)(\C{M})(F_1(z),\ldots,F_d(z))=d.$$ In particular, the functions $F_1(z),\ldots,F_d(z)$ are algebraically independent over $\B{C}(z)$.
\end{theorem}

Theorems \ref{main} and \ref{mainreg} have some interesting corollaries; we list three here. The first concerns the derivatives of regular functions.

\begin{corollary} Let $k\geqslant 2$ be an integer and $F(z)$ be a $k$-regular function with $k$ and $\alpha_f$ multiplicatively independent. If $n_1$ and $n_2$ are any two distinct nonnegative integers, then $$\trdeg_{\B{C}(z)}\B{C}(z)(e^z,F^{(n_1)}(z),F^{(n_2)}(z))=3,$$ where $F^{(n)}(z)$ denotes the $n$th derivative of $F(z)$.
\end{corollary}

The next corollary demonstrates that Theorems \ref{main} and \ref{mainreg} can be used to give results for infinite sets of functions.

\begin{corollary}\label{DN} Let $p$ be an odd prime, let $\Phi_p(z)$ be the $p$th cyclotomic polynomial, let $k\geqslant 2$ be an integer, and set $F_p(z):=\prod_{n\geqslant 0}\Phi_p(z^{k^n}).$ Then the functions $$F_3(z),F_5(z),F_7(z),\ldots,F_p(z),\ldots,$$ with indices odd primes $p$ coprime to $k$, are algebraically independent over $\B{C}(z)$.
\end{corollary}

\noindent The functions considered in Corollary \ref{DN} were recently studied by Duke and Nguyen~\cite{DN2015}.

Our last corollary in this Introduction concerns the algebraic independence of Mahler functions of different degrees. 

\begin{corollary}\label{FS} Let $S(z)$ be Stern's function satisfying $zS(z)-(1+z+z^2)S(z^2)=0$, and $F(z)$ be the function of Dilcher and Stolarsky \cite{DS2009}, which has $0,1$-coefficients and satisfies $F(z)-(1+z+z^2)F(z^4)+z^4F(z^{16})=0.$ Then $$\trdeg_{\B{C}(z)}\B{C}(z)(S(z), F(z))=2.$$
\end{corollary}

Corollary \ref{FS} holds also for any pair of derivatives of $S(z)$ and $F(z)$. We note that the algebraic independence over $\B{C}(z)$ of the Dilcher-Stolarsky function $F(z)$ and its derivative $F'(z)$ follows from our recent joint work with Brent and Zudilin \cite{BCZ2015}.

The remainder of this paper is organised as follows. In Section \ref{SecMahl}, we prove Theorem \ref{main}. Section \ref{SecReg} contains the proofs of Theorems \ref{fNlogN} and \ref{mainreg}. In the final section, we present an extended `illustrative' example as well as a few corollaries and questions.

\section{Algebraic independence of Mahler functions}\label{SecMahl}

In this section, we prove Theorem \ref{main}. Our proof relies heavily on the use of the radial asymptotics of Mahler functions as $z$ approaches various roots of unity. In joint work with Bell, we recently provided the initial case of the more general result to follows here (see Theorem \ref{xi}). We record the special case here as a proposition.

\begin{proposition}[Bell and Coons \cite{BCpre}]\label{initial} Let $F(z)$ be a $k$-Mahler function satisfying \eqref{MFE} whose characteristic polynomial $p_F(\lambda)$ has $d$ distinct roots. Then there is an eigenvalue $\lambda_F$ with $p_F(\lambda_F)=0$, such that as $z\to 1^-$ \begin{equation}\label{Fzto1}F(z)=\frac{C_F(z)}{(1-z)^{\log_k \lambda_F}} (1+o(1)),\end{equation} where $\log_k$ denotes the principal value of the base-$k$ logarithm and $C_F(z)$ is a real-analytic nonzero oscillatory term, which on the interval $(0,1)$ is bounded away from $0$ and $\infty$, and satisfies $C_F(z)=C_F(z^k)$. 
\end{proposition} 

For the purposes of transcendence, Proposition \ref{initial} is enough; this was the purpose of our joint work with Bell \cite{BCpre}. To gain algebraic independence results, we additionally require the asymptotics as $z$ approaches a general root of unity of degree $k^n$ for any $n\geqslant 0$. Concerning these asymptotics, we give the following result.

\begin{theorem}\label{xi} Let $F(z)$ be a $k$-Mahler function satisfying \eqref{MFE} whose characteristic polynomial $p_F(\lambda)$ has $d$ distinct roots and let $\xi$ be a root of unity of degree $k^n$ for some $n\geqslant 0.$ Then as $z\to 1^-$, there is an integer $m_\xi$ and a nonzero number $\Lambda_F(\xi)$ such that $$F(\xi z)=\frac{\Lambda_F(\xi)C_F(z)}{(1-z)^{\log_k \lambda_F-m_\xi}}(1+o(1)),$$ where $C_F(z)$ is the function of Theorem \ref{initial}.
\end{theorem}

\begin{proof} If $\xi_1$ is a root of unity of degree $k$, then using the functional equation \eqref{MFE} and Proposition \ref{initial}, \begin{align} \nonumber F(\xi_1 z)&=\frac{-1}{a_0(\xi_1 z )}\sum_{j=1}^d a_j(\xi_1 z)F(z^{k^j})\\
\label{ratxiz}&=\left(-\sum_{j=1}^d \frac{a_j(\xi_1 z)}{a_0(\xi_1z)}\lambda_F^{-j}\right)\frac{C_F(z)}{(1-z)^{\log_k \lambda_F}}(1+o(1))\\
\nonumber &=\frac{\Lambda_F(\xi_1)C_F(z)}{(1-z)^{\log_k \lambda_F-m_1}}(1+o(1)),
\end{align} where we have used the fact that as $z\to 1^-$, the rational function in \eqref{ratxiz} can be written $$-\sum_{j=1}^d \frac{a_j(\xi_1 z)}{a_0(\xi_1z)}\lambda_F^{-j}=\Lambda_F(\xi_1)(1-z)^{m_1}(1+o(1)),$$ for some nonzero complex number $\Lambda_F(\xi_1)$ and some integer $m_1$ that depends on $\xi_1$ and the polynomials $a_i(z)$, for $i=0,\ldots,d$. The fact that $\sum_{j=1}^d \frac{a_j(\xi_1 z)}{a_0(\xi_1z)}\lambda_F^{-j}\neq 0$ follows from the assumption that the characteristic polynomial $p_F(\lambda)$ has $d$ distinct roots.

Note that we can continue this process iteratively for any root of unity $\xi$ of degree $k^n$, as $z\to \xi$ radially.

The result is now a direct consequence of the above argument with the additional realisation that for roots of unity $\xi$ of degree $k^n$ with $n$ large enough, $a_i(\xi)\neq 0$ for $i=0,\ldots,d$.
\end{proof}

For a special case of Theorem \ref{xi}, see our recent work with Brent and Zudilin \cite[Theorem 3 and Lemma 5]{BCZ2015}, in which we used radial asymptotics to extend the work of Bundschuh and V\"a\"an\"anen \cite{BV2014}. 

With these asymptotic results established, we may now prove Theorem \ref{main}.

\begin{proof}[Proof of Theorem \ref{main}] Towards a contradiction, assume the theorem is false, so that we have an algebraic relation \begin{equation}\label{ar}\sum_{{\bf m}=(m_1,\ldots,m_d)\in{\bf M}}p_{\bf m}(z,G_1(z),\ldots,G_s(z))F_1(z)^{m_1}\cdots F_d(z)^{m_d}=0,\end{equation} where the set ${\bf M}\subseteq \B{Z}^{d}_{\geqslant 0}$ is finite and none of the polynomials $p_{\bf m}(z,G_1(z),\ldots,G_s(z))$ in $\B{C}[z][\C{M}]$ is identically zero. Moreover, without loss of generality, we may suppose that the polynomial $\sum_{{\bf m}}p_{\bf m}(z,w_1\ldots,w_s)y_1^{m_1}\cdots y_d^{m_d}$ in $d+s+1$ variables is irreducible.

Pick a $z_0\in(0,1)$ and note that as $z\to 1^-$ along the sequence $\{z_0^{k^m}\}_{m\geqslant 0}$, for $\xi$ any root of unity of degree $k^n$, with $n$ large enough, in the notation of Theorem \ref{xi}, we have \begin{equation}\label{algterm}F_1(\xi z)^{m_1}\cdots F_d(\xi z)^{m_d}= \frac{C_{{\bf m}} \left(\prod_{i=1}^d \Lambda_{F_i}(\xi)^{m_i}\right)}{(1-z)^{m_1\log_k\lambda_{F_1}+\cdots+m_d\log_k\lambda_{F_d}+|{\bf m}|\cdot m}}(1+o(1)),\end{equation} where $|{\bf m}|=m_1+\cdots+m_d$, $m$ is an integer, and $C_{\bf m}\neq 0$ depends on the choice of $z_0$, but is independent of $\xi$ and $z$.

Let ${\bf M}_{\max}\subseteq{\bf M}$ be the (nonempty) set of indices ${\bf m}=(m_1,\ldots,m_d)$ such that the quantity $$\delta:= m_1\log_k\lambda_{F_1}+\cdots+m_d\log_k\lambda_{F_d}+|{\bf m}|\cdot m$$ is maximal. 

We claim that the set ${\bf M}_{\max}$ contains only one element. To see this, suppose that ${\bf m},{\bf m}'\in{\bf M}_{\max}.$ Then $$m_1\log_k\lambda_{F_1}+\cdots+m_d\log_k\lambda_{F_d}+|{\bf m}|\cdot m=m_1'\log_k\lambda_{F_1}+\cdots+m_d'\log_k\lambda_{F_d}+|{\bf m}'|\cdot m=\delta,$$ and so $$(m_1-m_1')\log_k\lambda_{F_1}+\cdots+(m_d-m_d')\log_k\lambda_{F_d}+(|{\bf m}|-|{\bf m}'|)m=0.$$ Since the numbers $k,\lambda_{F_1},\ldots,\lambda_{F_d}$ are multiplicatively independent, the numbers $\log_k\lambda_{F_1},\ldots,\log_k\lambda_{F_d},m$ are linearly independent. Thus $m_i=m_i'$ for each $i\in\{1,\ldots,d\}$, and we have ${\bf m}={\bf m}'.$

Using the uniqueness of the term of index ${\bf m}_{\max}$ with maximal asymptotics, we multiply the algebraic relation \eqref{ar} by $(1-z)^\delta$ and send $z\to 1^-$ along the sequence $\{z_0^{k^m}\}_{m\geqslant 0}$. Then for a root of unity $\xi$ of degree $k^n$, for $n$ large enough and for which $G_1(\xi),\ldots,G_s(\xi)$ each exist, we gain the equality $$p_{{\bf m}_{\max}}(\xi,G_1(\xi)\ldots,G_s(\xi))\cdot C_{{\bf m}_{\max}}\cdot \left(\prod_{i=1}^d \Lambda_{F_i}(\xi)^{m_i}\right)=0.$$ This implies that $$p_{{\bf m}_{\max}}(\xi,G_1(\xi)\ldots,G_s(\xi))=0,$$ for each choice of such $\xi$. Since there are infinitely many such $\xi$ that are dense on the unit circle and $p_{{\bf m}_{\max}}(z,G_1(z),\ldots,G_s(z))$ is a meromorphic function, it must be that $p_{{\bf m}_{\max}}(z,G_1(z),\ldots,G_s(z))=0$ identically, contradicting our original assumption.
\end{proof}

\section{Algebraic independence of regular functions}\label{SecReg}

In this section, we prove Theorems \ref{fNlogN} and \ref{mainreg}. As stated in the Introduction, focusing on the subclass of nonnegative regular functions allows us a bit more freedom in the results. While we will still use the fact that regular functions $F(z)$ are Mahler functions, we no longer require the existence of the eigenvalue $\lambda_F$. For nonnegative regular functions, the role of the eigenvalue will be played by a different constant $\alpha_f$, some properties of which are discussed in what follows. 

To establish Theorem \ref{fNlogN}, we require a few preliminary results, the first of which separates out a special linear recurrent subsequence of a regular sequence.

\begin{lemma}\label{sigma} If $f$ is a $k$-regular sequence, then $\{f(k^\ell)\}_{\ell\geqslant 0}$ is linearly recurrent. 
\end{lemma}

\begin{proof} Recalling the definition of regular sequences in the Introduction, let $\{{\bf A}_0,\ldots,$ ${\bf A}_{k-1}\}$, ${\bf v}$, and ${\bf w}$ be such that $f(m)={\bf w}^T {\bf A}_{i_0}\cdots{\bf A}_{i_s} {\bf v},$ where $(m)_k={i_s}\cdots {i_0}$ is the base-$k$ expansion of $m$. Then we have $$f(k^\ell)={\bf w}^T {\bf A}_{0}^{\ell}{\bf A}_{1} {\bf v},$$ which proves the lemma.
\end{proof}

Though unneeded for our purposes, it is worth noting that one may strengthen the above lemma to show that for any choice of $n$ and $r$, the sequence $\{f(k^\ell n+r)\}_{\ell\geqslant 0}$ is linearly recurrent. 

We require the following classical result of Allouche and Shallit \cite[Theorem 3.1]{AS1992}. 

\begin{proposition}[Allouche and Shallit \cite{AS1992}]\label{AS} Let $k\geqslant 2$ be an integer. Then the set of $k$-regular sequences is closed under (Cauchy) convolution. In particular, if $\{f(n)\}_{n\geqslant 0}$ is $k$-regular, then so is the sequence $\{g(n)\}_{n\geqslant 0}$, where $$g(n)=\sum_{j\leqslant n}f(j).$$
\end{proposition}

By applying Lemma \ref{sigma} and Proposition \ref{AS}, we now prove Theorem \ref{fNlogN}.

\begin{proof}[Proof of Theorem \ref{fNlogN}] Combining Lemma \ref{sigma} and Proposition \ref{AS}, we have that $\sigma_f(r):=\sum_{n\leqslant k^r} f(n)$ is linearly recurrent. Further, since $\{f(n)\}_{n\geqslant 0}$ is nonnegative, the sequence $\{\sigma_f(r)\}_{r\geqslant 0}$ is increasing. Thus using the eigenvalue representation of the linear recurrence $\sigma_f(r)$, as $r\to\infty$, we have $$\sigma_f(r)=c_1r^{m_f}\alpha_f^r(1+o(1)),$$ for some integer $m_f\geqslant 0$ and $\alpha_f\geqslant 1$. The lower bound on $\alpha_f$ follows as $\sigma_f(r)$ is increasing and integer-valued.

The result of the theorem now follows quite quickly. To see this, let $N$ be large enough. Then $N\in(k^{r},k^{r+1}]$ for $r=\lfloor\log_k N\rfloor$. So for any $\eps>0$, $$(c_1-\eps)r^{m_f}\alpha_f^r<\sigma_f(r)\leqslant\sum_{n\leqslant N} f(n)\leqslant\sigma_f(r+1)<(c_1+\eps)(r+1)^{m_f}\alpha_f^{r+1}.$$ Using the trivial upper and lower bounds $\log_k N-1\leqslant r<\log_k N,$ we then have \begin{multline*}(c_1-\eps)\alpha_f^{-1}\left(1-\frac{1}{\log_k N}\right)^{m_f} N^{\log_k \alpha_f}\log_k^{m_f} N\\ \leqslant (c_1-\eps)r^{m_f}\alpha_f^r<\sum_{n\leqslant N} f(n)<(c_1+\eps)(r+1)^{m_f}\alpha_f^{r+1}\\ <(c_1+\eps)\alpha_f\left(1+\frac{1}{\log_k N}\right)^{m_f}N^{\log_k \alpha_f}\log_k^{m_f}N,\end{multline*} which finishes the proof of the lemma.
\end{proof}

While the statement of Theorem \ref{fNlogN} is very precise, we use it in a less technical way; we need only the fact that $$\sum_{n\leqslant N} f(n)\asymp N^{\log_k\alpha_f} \log^{m_f}N.$$

In order to prove Theorem \ref{mainreg}, we determine an asymptotic result that mimics Theorem \ref{xi} for nonnegative regular functions, making sure to avoid the need of a Mahler eigenvalue. As in the proof of Theorem \ref{fNlogN} above, nonnegativity remains an important assumption.

\begin{proposition} Let $k\geqslant 2$ be an integer and $\{f(n)\}_{n\geqslant 0}$ be a nonnegative integer-valued $k$-regular sequence, which is not eventually zero. Let $\alpha_f\geqslant 1$ be as given by Theorem \ref{fNlogN}. If $F(z)=\sum_{n\geqslant 0}f(n)z^n$, then for any $\eps>0$, as $z\to 1^-$, $$\frac{1}{(1-z)^{\log_k\alpha_f+\eps}}\leqslant F(z)\leqslant \frac{1}{(1-z)^{\log_k\alpha_f-\eps}}.$$
\end{proposition}

\begin{proof} Let $k\geqslant 2$ be a integer and $\{f(n)\}_{n\geqslant 0}$ be a nonnegative integer-valued $k$-regular sequence, which is not eventually zero. Set $$G(z):=\frac{F(z)}{1-z}.$$ Let $\alpha_f\geqslant 1$ and $m_f\geqslant 0$ be as given in Theorem \ref{fNlogN}. For $z\in(0,1)$ define the function $$C(z):=\frac{G(z)}{1+\sum_{n\geqslant 1} n^{\log_k\alpha_f}(\log_k n)^{m_f} z^n}.$$ By Theorem \ref{fNlogN} and the fact that the series in the denominator is nonzero and differentiable on $(0,1)$, we have that on $(0,1)$ the function $C(z)$ is nonzero, differentiable, and bounded above and below by positive constants.

Set $$D(z):=1+\sum_{n\geqslant 1} n^{\log_k\alpha_f}(\log_k n)^{m_f} z^n.$$ We continue by finding asymptotic bounds on the function $D(z)$. To this end, note that for any positive real number $r$, we have $$\frac{1}{(1-z)^r}
=\sum_{n\geqslant 0}\frac{\gG(r+n)}{\gG(r)n!}z^n,$$ where $\gG(z)$ is the Euler $\gG$-function. By Sterling's formula, we have that $$
\frac{\gG(r+n)}{\gG(r)n!}\sim \frac{n^{r-1}}{\Gamma(r)}.$$ It then follows from a classical result of C\'esaro (see P\'olya and Szeg\H{o} \cite[Problem 85 of Part I]{PS1}, that for any given $\eps>0$, as $z\to 1^-$, \begin{equation}\label{Dmajmin}\frac{1}{(1-z)^{\log_k\alpha_f+1+\eps}}\leqslant D(z)\leqslant \frac{1}{(1-z)^{\log_k\alpha_f+1-\eps}}.\end{equation}

Recall $C(z)\asymp 1$, so we have $G(z)\asymp D(z)$, and thus \eqref{Dmajmin} holds with $D(z)$ replaced by $G(z)$. The result now follows since $F(z)=(1-z)G(z)$.
\end{proof}

In order to simplify our further exposition, we make the following definition.

\begin{definition} We call a function $H(z)$ an {\em $\eps$-function}, if there is an $a>0$ such that $H(z)$ is defined on the interval $(1-a,1)$, and as $z\to 1^-$, either $H(z)$ is bounded away from zero and infinity, or the function satisfies $H(z)=o((1-z)^{\eps})$ or $H(z)=o((1-z)^{-\eps})$ for any $\eps>0$.
\end{definition}

\begin{corollary}\label{FL} Let $k\geqslant 2$ be an integer and $\{f(n)\}_{n\geqslant 0}$ be a nonnegative integer-valued $k$-regular sequence, which is not eventually zero. Let $\alpha_f\geqslant 1$ be as given by Theorem \ref{fNlogN}. If $F(z)=\sum_{n\geqslant 0}f(n)z^n$, then there is an $\eps$-function $L(z)$ such that $$F(z)=\frac{L(z)}{(1-z)^{\log_k\alpha_f}}(1+o(1)),$$ as $z\to 1^-$.
\end{corollary}

We now extend Corollary \ref{FL} to include all  radial limits as $z$ approaches a root of unity of degree $k^n$ for $n$ large enough. In this way, the following result is the analogue of Theorem \ref{xi} for nonnegative regular functions.

\begin{theorem}\label{regxizto1} Let $k\geqslant 2$ be an integer and $\{f(n)\}_{n\geqslant 0}$ be a nonnegative integer-valued $k$-regular sequence, which is not eventually zero. Let $\alpha_f\geqslant 1$ be as given by Theorem~\ref{fNlogN} and set $F(z)=\sum_{n\geqslant 0}f(n)z^n$. If $\xi$ is a root of unity of degree $k^n$, with $n$ large enough, then there is an integer $d_\xi$ and an $\eps$-function $L_\xi(z)$ such that $$F(\xi z)=\frac{L_\xi(z)}{(1-z)^{\log_k\alpha_f+d_\xi}}(1+o(1)),$$ as $z\to 1^-$.
\end{theorem}

\begin{proof} By Corollary \ref{FL}, there is a real number $\alpha_f\geqslant 1$ and an $\eps$-function $L_0(z)$ such that as $z\to 1^-$, we have $$F(z)=\frac{L_0(z)}{(1-z)^{\log_k\alpha_f}}(1+o(1)).$$ 

Recall that any $k$-regular function is a $k$-Mahler function; using this fact, let us suppose that $F(z)$ satisfies \eqref{MFE}. 

Let $\xi_1$ be a $k$-th root of unity. Then as $z\to 1^-$, we have $$ F(\xi_1 z)=-\sum_{i=1}^d \frac{a_i(\xi_1 z)}{a_0(\xi_1 z)}F(x^{k^i})=\frac{-\sum_{i=1}^d\frac{a_i(\xi_1 z)}{a_0(\xi_1 z)}\alpha_f^{-i}L_0(x^{k^i})}{(1-z)^{\log_k\alpha_f}}(1+o(1)).$$ Note that since $L_0(z)$ is an $\eps$-function, so is $L_0(z^{k^i})$. Thus as $z\to 1^-$, we have $$-\sum_{i=1}^d\frac{a_i(\xi_1 z)}{a_0(\xi_1 z)}\alpha_f^{-i}L_0(x^{k^i})=(1-z)^{d_1}L_1(z)(1+o(1)),$$ for some $\eps$-function $L_1(z)$ and some integer $d_1\in\B{Z}$, which depends on the polynomials $a_i(z)$ ($i=0,\ldots,d$), $\alpha_f$, and $\xi_1$. Thus $$F(\xi_1 z)=\frac{L_1(z)}{(1-z)^{\log_k \alpha_f-d_1}}(1+o(1)),$$ as $z\to 1^-$. 

Continuing in this way, if $\xi$ is a root of unity of degree $k^n$, with $n$ large enough so that $a_i(\xi)\neq 0$ for each $i=0,\ldots,d$, then there is an $\eps$-function $L_\xi(z)$ and an integer $d_\xi$, such that $$F(\xi z)=\frac{L_\xi(z)}{(1-z)^{\log_k\alpha_f+d_\xi}}(1+o(1))$$ as $z\to 1^-$, which is the desired result.
\end{proof}

With our asymptotic results in place, we can now prove Theorem \ref{mainreg}.

\begin{proof}[Proof of Theorem \ref{mainreg}] This proof follows very close our proof of Theorem \ref{main}, but with the $\eps$-functions $L(z)$ given by Theorem \ref{regxizto1} in place of the functions $C(z)$ from Theorem \ref{xi}. 

To start, as in our proof of Theorem \ref{main}, and towards a contradiction, assume the theorem is false, so that we have an algebraic relation \begin{equation}\label{arreg}\sum_{{\bf m}=(m_1,\ldots,m_d)\in{\bf M}}p_{\bf m}(z,G_1(z),\ldots,G_s(z))F_1(z)^{m_1}\cdots F_d(z)^{m_d}=0,\end{equation} where the set ${\bf M}\subseteq \B{Z}^{d}_{\geqslant 0}$ is finite and none of the polynomials $p_{\bf m}(z,G_1(z),\ldots,G_s(z))$ in $\B{C}[z][\C{M}]$ is identically zero. Again, without loss of generality, we may suppose that the polynomial $\sum_{{\bf m}}p_{\bf m}(z,w_1\ldots,w_s)y_1^{m_1}\cdots y_d^{m_d}$ in $d+s+1$ variables is irreducible.

As $z\to 1^-$, for $\xi$ any root of unity of degree $k^n$, with $n$ large enough, in the notation of Theorem \ref{regxizto1}, we have \begin{equation}\label{algterm}F_1(\xi z)^{m_1}\cdots F_d(\xi z)^{m_d}=\frac{ L_{\xi, {\bf m}}(z)}{(1-z)^{m_1\log_k\lambda_{F_1}+\cdots+m_d\log_k\lambda_{F_d}+|{\bf m}|\cdot m}}(1+o(1)),\end{equation} where $|{\bf m}|=m_1+\cdots+m_d$, $m$ is an integer, and $$L_{\xi,{\bf m}}(z):=\prod_{i=1}^d L_{F_i,\xi}(z)^{m_i}$$ is an $\eps$-function. Note that $L_{\xi,{\bf m}}(z)$ is an $\eps$-function since a product of $\eps$-functions is again an $\eps$-function. 

Let ${\bf M}_{\max}\subseteq{\bf M}$ be the (nonempty) set of indices ${\bf m}=(m_1,\ldots,m_d)$ such that the quantity $$\delta:= m_1\log_k\lambda_{F_1}+\cdots+m_d\log_k\lambda_{F_d}+|{\bf m}|\cdot m$$ is maximal. Note here that the asymptotic properties of the function $L_{\xi,{\bf m}}(z)$ do not effect the maximal asymptotics in a way that changes the set ${\bf M}_{\max}$.

Again, we claim that the set ${\bf M}_{\max}$ contains only one element. To see this, suppose that ${\bf m},{\bf m}'\in{\bf M}_{\max}.$ Then $$m_1\log_k\lambda_{F_1}+\cdots+m_d\log_k\lambda_{F_d}+|{\bf m}|\cdot m=m_1'\log_k\lambda_{F_1}+\cdots+m_d'\log_k\lambda_{F_d}+|{\bf m}'|\cdot m=\delta,$$ and so $$(m_1-m_1')\log_k\lambda_{F_1}+\cdots+(m_d-m_d')\log_k\lambda_{F_d}+(|{\bf m}|-|{\bf m}'|)m=0.$$ Since the numbers $k,\lambda_{F_1},\ldots,\lambda_{F_d}$ are multiplicatively independent, the numbers $\log_k\lambda_{F_1},\ldots,\log_k\lambda_{F_d},m$ are linearly independent. Thus $m_i=m_i'$ for each $i\in\{1,\ldots,d\}$ and ${\bf m}={\bf m}'.$

Using the uniqueness of the term of index ${\bf m}_{\max}$ with maximal asymptotics, we multiply the algebraic relation \eqref{ar} by $(1-z)^\delta/L_{\xi,{\bf m}_{\max}}(z)$ and send $z\to 1^-$. Then for a root of unity $\xi$ of degree $k^n$, for $n$ large enough and for which $G_1(\xi),\ldots,G_s(\xi)$ each exist,  we gain the equality $$p_{{\bf m}_{\max}}(\xi,G_1(\xi),\ldots,G_s(\xi))=0.$$ Since there are infinitely many such $\xi$ which are dense on the unit circle, and $p_{{\bf m}_{\max}}(z,G_1(z),\ldots,G_s(z))$ is a meromorphic function, it must be that $$p_{{\bf m}_{\max}}(z,G_1(z),\ldots,G_s(z))=0$$ identically, contradicting our original assumption.
\end{proof}

\section{Concluding remarks}\label{SecCon}

In this paper, we presented general algebraic independence results for Mahler functions and nonnegative regular functions. We did this by making use of the arithmetic properties of the eigenvalues of Mahler functions, and by the arithmetic properties of certain constants, which we denoted $\alpha_f$, associated to regular sequences $\{f(n)\}_{n\geqslant 0}$. 

Before ending our paper, we provide an extended example that illustrates some of the objects that we have considered as well as a few corollaries and questions that may be of interest.

\subsection{An extended example: Stern's sequence} Let $\{s(n)\}_{n\geqslant 0}$ be {\em Stern's diatomic sequence}. Stern's sequence is determined by the relations $s(0)=0$, $s(1)=1$, and for $n\geqslant 0$, by $$s(2n)=s(n), \quad\mbox{and}\quad s(2n+1)=s(n)+s(n+1).$$ This sequence is $2$-regular and is determined by the vectors and matrices $${\bf w}={\bf v}=[1\ 0]^T\quad\mbox{and}\quad ({\bf A}_0,{\bf A}_1)=\left(\left[\begin{array}{rr} 1&1\\ 0&1\end{array}\right],\left[\begin{array}{rr} 1&0\\ 1&1\end{array}\right]\right).$$ Further, the generating function $S(z)=\sum_{n\geqslant 0}s(n)z^n$ is $2$-Mahler and satisfies the functional $$zS(z)-(1+z+z^2)S(z^2)=0.$$ The characteristic polynomial for $S(z)$ is linear; it is $p_S(\lambda)=\lambda-3.$ Thus $\lambda_S=3$. Of course, since $S(z)$ is regular, the constant $\alpha_s$ also exists, and in this case $\alpha_s=\lambda_S=3.$ In fact, this equality holds for all regular functions; that is, if $F(z)=\sum_{n\geqslant 0}f(n)z^n$ is regular and $\lambda_F$ exists, then $\lambda_F=\alpha_f$. 

To illustrate the effect of these constants, we use two figures. In Figure \ref{Stern15}, we have plotted the values of the Stern sequence in the interval $[2^{15}, 2^{16}]$. As the power of two is increased this picture will fill out with more values, but already at these values it is quite stable. 

\begin{figure}[htp]
\includegraphics[width=4.7in,height=2.5in]{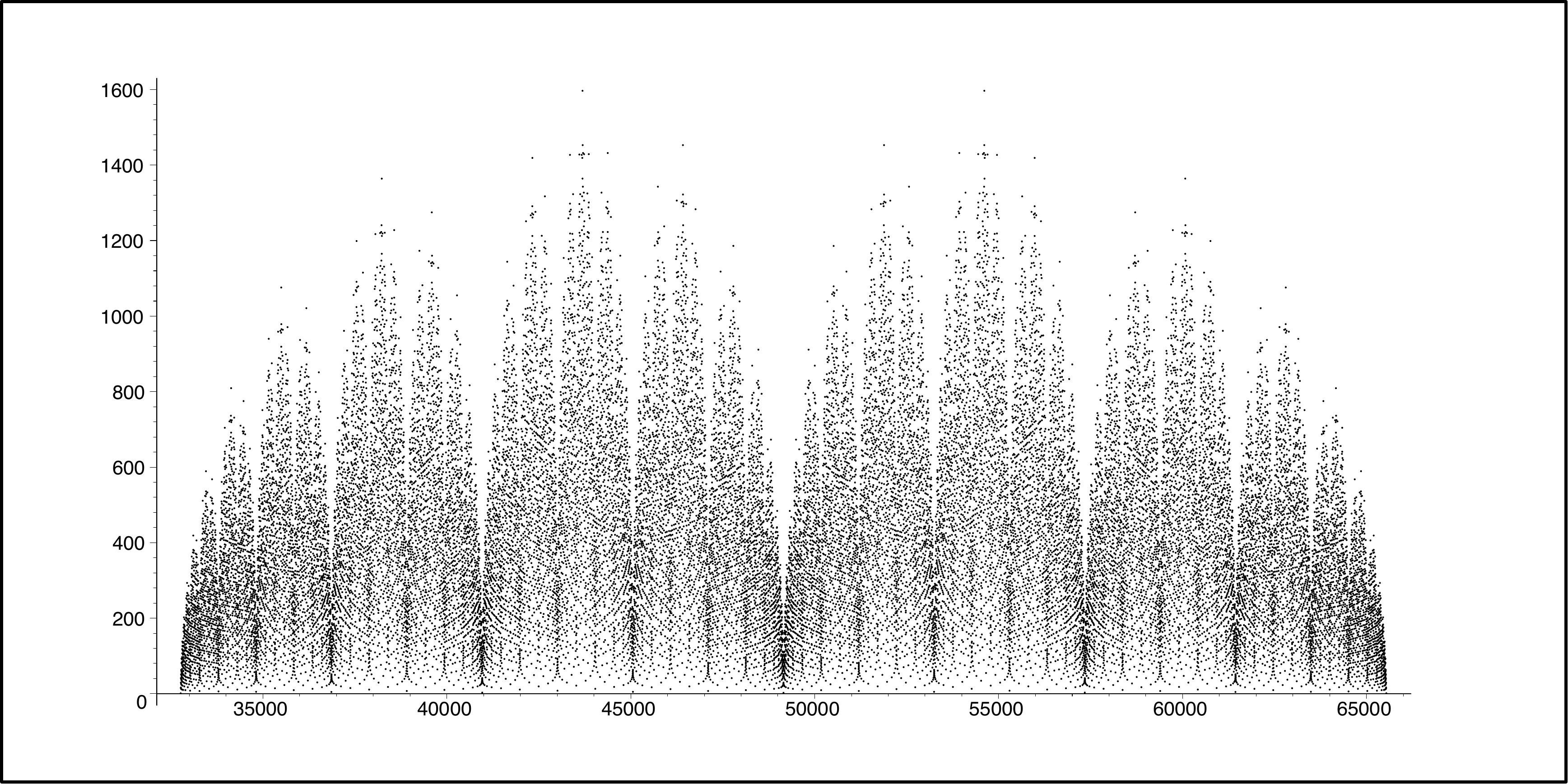}
\caption{Stern's diatomic sequence in the interval $[2^{15}, 2^{16}]$.}
\label{Stern15}
\end{figure}

\noindent Notice that the Stern sequence, while definitely exhibiting structure is quite erratic as well. But if we consider the weighted partial sums $N^{-\log_2 3}\sum_{n\leqslant N}s(n)$ for $N$ between two large powers of two, a whole other structure arises; see Figure \ref{Sternperiod}.

\newpage

\begin{figure}[htp]
\includegraphics[width=4.7in,height=2.5in]{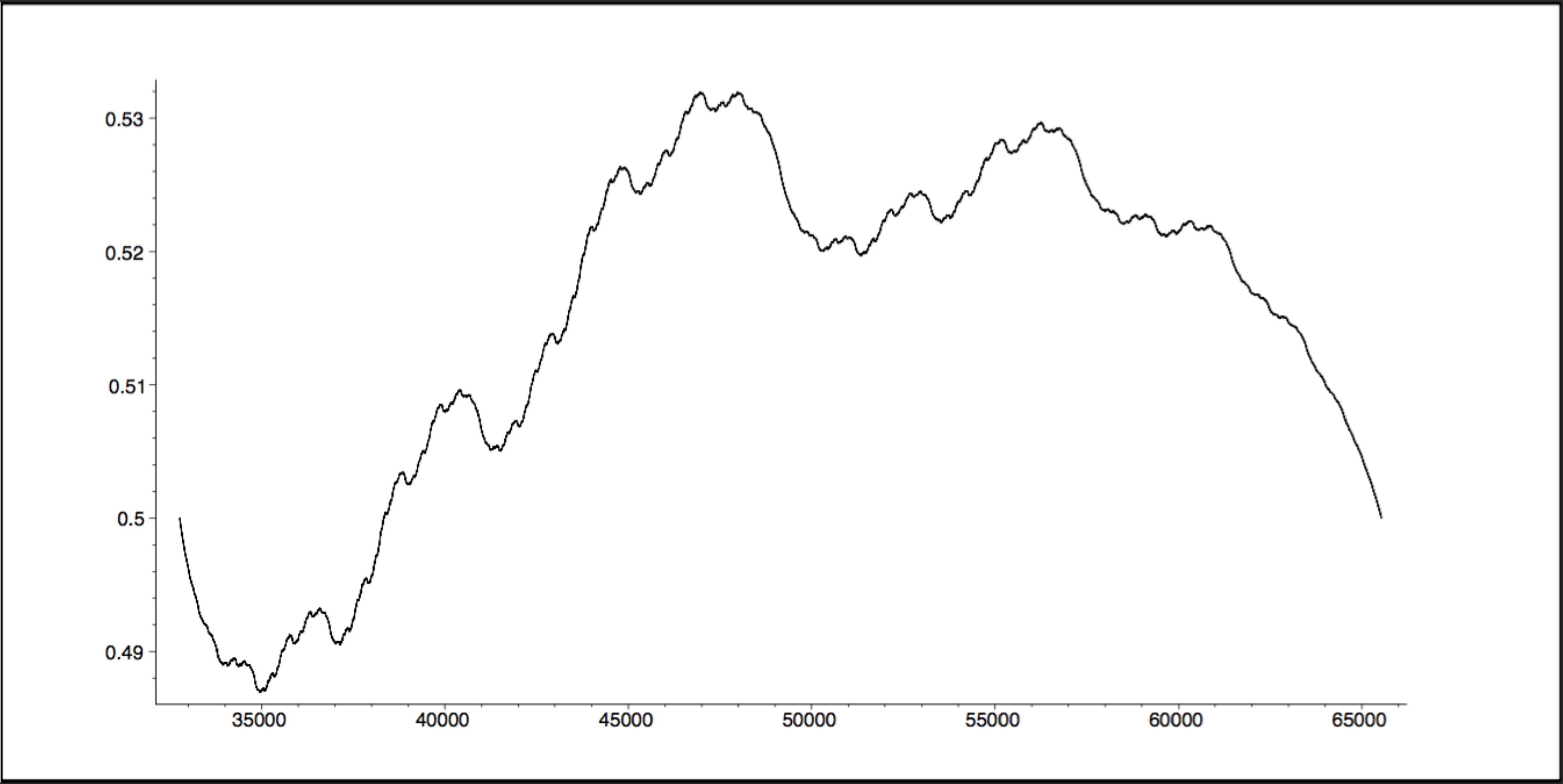}
\caption{The weighted partial sums $N^{-\log_2 3}\sum_{n\leqslant N}s(n)$ in the interval $[2^{15}, 2^{16}]$.}
\label{Sternperiod}
\end{figure}

\noindent Figure \ref{Sternperiod} illustrates Theorem \ref{fNlogN}. In the case of the Stern sequence, the upper and lower (asymptotic) bounds given are $\alpha_s=3$ and $\alpha_s^{-1}=1/3,$ though it is quite clear from Figure \ref{Sternperiod} that these bounds are not optimal.

Many results have been proven regarding the transcendence and algebraic independence of the generating function of the Stern sequence, so that providing interesting corollaries of Theorems \ref{main} and \ref{mainreg} using this function is a bit of a challenge. We proved \cite{C2010} the transcendence of $S(z)$, and Bundschuh \cite{B2012} extended that result by showing that the derivatives $$S(z),S'(z),S^{(2)}(z),\ldots,S^{(n)}(z),\ldots$$ are algebraically independent over $\B{C}(z)$. But if we throw a wrench in the works, things get more interesting. We can pair the Stern function with other functions and give results which previous methods could not attack. Corollary \ref{FS} stated in the Introduction is a good example of this. We give another example here. 

The {\em Baum-Sweet sequence} is given by the recurrences $b_0=1,$ $b_{4n}=b_{2n+1}=b_n$, and $b_{4n+2}=0.$ The sequence $\{b_n\}_{n\geqslant 0}$ is $2$-automatic (and so also $2$-regular) and as a consequence of the above relations, its generating function $B(z)=\sum_{n\geqslant 0}b_n z^n$ satisfies the $2$-Mahler equation $$B(z)-zB(z^2)-B(z^4)=0.$$ The function $B(z)$ has the characteristic polynomial $p_B(\lambda)=\lambda^2-\lambda-1$ and $\lambda_B=(1+\sqrt{5})/2$. We thus have the following corollary to Theorem \ref{main}.

\begin{corollary} For $S(z)$ the Stern's function and $B(z)$ the Baum-Sweet function, $$\trdeg_{\B{C}(z)}\B{C}(z)(S(z),B(z))=2.$$
\end{corollary}

\subsection{Hypertranscendence of Mahler functions} As mentioned in the Introduction, B\'ezivin \cite{B1994} proved that a Mahler function that is $D$-finite (that is, satisfies a homogeneous linear differential equation with polynomial coefficients) is rational. Of course, the general question of hypertranscendence of irrational Mahler functions remains open. Recall that a function is called {\em hypertranscendental} if it does not satisfy an algebraic differential equation with polynomial coefficients. We reiterate the following question.

\begin{question} Is it true that an irrational Mahler function is hypertranscendental?
\end{question}

Partial progress has been made in the case of Mahler functions of degree $1$ (see Bundschuh \cite{B2012, B2013}), though there is not a single known example of a hypertranscendental Mahler function of degree $2$ or greater.

Towards this question, we can offer the following modest result, which is a corollary to Theorems \ref{main} and \ref{mainreg}.

\begin{corollary}\label{deg3} Let $F(z)=\sum_{n\geqslant 0}f(n)z^n$ be either a Mahler function for which $\lambda_F$ exists and $k$ and $\lambda_F$ are multiplicatively independent, or a nonnegative regular function with $k$ and $\alpha_f$ multiplicatively independent. If $p(z,z_0,\ldots,z_n)\in\B{C}[z,z_0,\ldots,z_n]$ is any polynomial with $p(z,F(z),F'(z),\ldots,F^{(n)}(z))=0$, then for any $c\in\B{C}$ $$\deg p(c,z_0,\ldots,z_n)\geqslant 3.$$
\end{corollary}

\noindent Note that Corollary \ref{deg3} also holds with the addition of any number of meromorphic functions. In this way, it seems natural to consider as well the hypertranscendence of Mahler functions over polynomials in meromorphic functions.

\subsection{A question about $e$ and $\pi$ and Mahler numbers} We end this paper with one more corollary (more a novelty of our method) as well as a question, which we hope will stimulate further research using possibly a combination of algebraic and asymptotic techniques for algebraic independence.

\begin{corollary}\label{Szeta} If $F(z)$ is a $k$-Mahler function for which $\lambda_F$ exists and $k$ and $\lambda_F$ are multiplicatively independent, then $$\trdeg_{\B{C}(z)}\B{C}(z)(F(z),G(z),\zeta(z))=3,$$ where $G(z)$ is any $D$-finite function and $\zeta(z)$ is Riemann's zeta function.
\end{corollary}

\noindent The proof of this corollary, follows from the fact that $\trdeg_{\B{C}(z)(G(z),\zeta(z))}(F(z))=1$, and $\trdeg_{\B{C}(z)}(G(z),\zeta(z))=2$, which is implied by the work of Ostrowski \cite{O1920} and methods that can be found in Kaplansky's short book \cite{K1976}.

Corollary \ref{Szeta} can be viewed as a sort of classification, which is saying that Mahler functions, functions that satisfy linear differential equations, and the zeta function are very different sorts of functions. It would be extremely interesting interesting if one could prove an algebraic independence result about the special values of such functions. In particular, one would like to answer the following question.

\begin{question}\label{ez} Is it true for a $k$-Mahler function $F(z)$ that $\trdeg_{\B{Q}}\B{Q}(F(\alpha),e)=2$ and/or $\trdeg_{\B{Q}}\B{Q}(F(\alpha),\pi)=2$ for any reasonable choice of algebraic $\alpha$?
\end{question}

\bibliographystyle{amsplain}
\def\cprime{$'$}
\providecommand{\bysame}{\leavevmode\hbox to3em{\hrulefill}\thinspace}
\providecommand{\MR}{\relax\ifhmode\unskip\space\fi MR }
\providecommand{\MRhref}[2]{%
  \href{http://www.ams.org/mathscinet-getitem?mr=#1}{#2}
}
\providecommand{\href}[2]{#2}


\end{document}